\documentclass[reqno]{amsproc}
\usepackage{amsmath}
\usepackage{amsfonts}
\usepackage{mathrsfs}
\usepackage[dvips]{graphicx}
\usepackage{amssymb}
\usepackage{amsbsy}
\usepackage{amsthm}
\usepackage{graphics}
\usepackage{color}
\usepackage{textcomp}
\usepackage{verbatim}

\makeatletter
\@namedef{subjclassname@2010}{%
\textup{2010} Mathematics Subject
Classification} \makeatother

\DeclareMathOperator{\sigmap}{\sigma_{\textrm{ap}}}

\DeclareMathOperator{\hol}{Hol}

\numberwithin{equation}{section}

\newtheorem{theorem}{Theorem}[section]
\newtheorem*{theorem*}{Theorem}

\newtheorem*{prob*}{Problem}

\newtheorem{lemma}[theorem]{Lemma}

\theoremstyle{remark}

\newcommand*{\bscr}{\mathscr{B}}
\newcommand*{\hh}{\mathcal{H}}

\newcommand*{\natu}{\mathbb{N}}
\newcommand*{\mcal}{\mathcal{M}}

\newcommand*{\nul}{\mathscr{N}}
\newcommand*{\comp}{\mathbb{C}}

\newcommand*{\bou}{\boldsymbol{B}}

\newcommand{\disc}{\mathbb{D}}

\theoremstyle{definition}

\newcommand{\lein}{\mathcal{L}}
\newcommand{\mul}{\mathcal{M}_z}
\newcommand{\bspace}{\mathcal{B}}

\newcommand*{\ran}{\mathscr{R}}

\newcommand*{\zbb}{\mathbb{Z}}

\newcommand{\kol}{\mathbb{T}}

\begin{document}

   \title[Cowen-Douglas operators and analytic Continuation] {Cowen-Douglas operators and analytic continuation}
\author[P. Pietrzycki ]{Pawe{\l} Pietrzycki}



   \address{Wydzia{\l} Matematyki i Informatyki, Uniwersytet
Jagiello\'{n}ski, ul. {\L}ojasiewicza 6,
PL-30348 Krak\'{o}w}

   \email{pawel.pietrzycki@im.uj.edu.pl}

   \begin{abstract}
 In this paper, we study certain Banach spaces of analytic functions on which a left-invertible multiplication operator acts. It turns out that, under natural conditions, its left inverse is a Cowen-Douglas operator. We investigate how the analytic continuations of functions from an invariant subspace of this Cowen-Douglas operator relate to the spectrum of its restriction to that subspace.
   \end{abstract}

   \maketitle

   \section{Introduction}
 Let $\bspace$ be a Banach space. Denote by $\bou(\bspace)$ the Banach algebra of all bounded linear
 operators on $\bspace$.
Given $T\in\bou(\bspace)$, we write $\ran(T)$ and $\nul(T)$ for the range  and the kernel of $T$, respectively. Here $\bigvee$ stands for the (closed) linear span. By $\mathbb{C}$ we denote the
field of complex
numbers. The
symbols $\mathbb{Z}$ and
$\mathbb{N}$
 are reserved for the sets of integers and
positive integers
respectively.

In \cite{ARR98},
 A. Aleman, S. Richter, and W. T. Ross studied a Banach space $\bspace$ of analytic functions on $\disc$ which has the following properties:
\begin{align*}
    &\mul \bspace \subset \bspace,\\
    &\bspace \hookrightarrow \hol(\disc), \text{ the inclusion is  
continuous,}\\
& 1\in \bspace,\\
& \lein_\lambda \bspace \subset \bspace,
\\
&\sigma(\mul) = \overline{\disc}\label{warost}.
\end{align*}

where for $\lambda\in \disc$, the operator $\lein_\lambda\colon \bspace \to \bspace$ is given by
\begin{equation*}
    (\lein_\lambda f)(z) = \frac{f(z)-f(\lambda)}{z-\lambda}, \qquad z \in \disc.
\end{equation*}
This example include the classical Banach spaces of holomorphic functions in the unit disc: the Hardy space, the Bergman space and the 
Dirichlet space and the Besov space. They  related meromorphic continuations of the functions from  an $\lein$-invariant subspace $\mcal$ of $\bspace$ to the spectrum of $\lein|_\mcal$. 
In this general setting, they proved  that $\sigma(\lein)=\overline{\disc}$, 
$\sigma(\lein|_\mcal)\subset \overline{\disc}$, for $\mcal$ invariant subspace of $\lein$ and
    \begin{align*}
    \sigma_{\text{ap}}(\lein|_\mcal)\cap \disc = \sigma_{\text{p}}(\lein|_\mcal)\cap \disc  =\{a\in \disc \colon \frac{1}{1-az}\in \mcal \}.
\end{align*}
Moreover, under a regularity condition
on $\bspace$, they proved that
\begin{align*}
\sigma_{\text{ap}}(\lein|_\mcal)\cap\kol=\kol\setminus\{\frac{1}{\zeta}\colon \text{every } f\in\mcal &\text{ extends to be analytic} \\& \text{ in a neighborhood of } \zeta \}.
\end{align*}

For a connected open subset $\varOmega$ of $\comp$ and $n\in \natu$, the Cowen-Douglas class $\bscr_n(\varOmega)$
consists of the operators $T\in\bou(\bspace)$
which satisfy:
\begin{itemize}
    \item[(i)] $\varOmega\subseteq\sigma(T)=\{\omega\in\comp\colon T-\omega \text{ not invertible}\},$
    \item[(ii)] $\ran(T-\omega)=\bspace$ for $\omega\in\varOmega$,
    \item[(iii)] $\bigvee_{\omega\in\varOmega}\ker(T-\omega)=\bspace$,
    \item[(iv)] $\dim \ker (T-\omega)=n$ for $\omega\in\varOmega$.
\end{itemize}
Cowen-Douglas operators have played an important role in operator theory, servicing as a bridge
between operator theory and complex geometry. 
Namely, for an operator $T\in\bscr_n(\varOmega)$ the mapping $\omega\to(T-\omega)$
gives rise to a Hermitian holomorphic vector bundle over $\varOmega$. Let us recall this notion. Let $\Lambda$ be a manifold with a complex structure and $n$ be a
positive integer. A rank \textit{$n$ holomorphic vector bundle} over $\Lambda$ consists of a manifold $E$ with
a complex structure together
with a holomorphic map $\pi$ from $E$ onto $\Lambda$ such that each
fibre $E_\lambda:=\pi^{-1}(\lambda)$  is isomorphic to $\comp^n$ and such that for each $\lambda_0$ in $\Lambda$ there exists a neighborhood $U$ of $\lambda_0$ and holomorphic functions $\gamma_1$,$\gamma_2$,...,$\gamma_n$ from $U$ to $E$ whose values form
a basis for $E_\lambda$ at each $\lambda\in\Lambda$. For an operator $T\in\bscr_n(\varOmega)$ let $(E_T,\pi)$ denote the subbundle of the trivial bundle $\varOmega\times\hh$ defined by
\begin{align*}
E_T=\{(w,x)\in\varOmega\times\bspace\colon x\in \nul(T-\omega)\} \text{ and } \pi(\omega,x)=\omega.
\end{align*}

\section{Main results}
  Let $\bspace$ be a Banach space of $E$-valued analytic functions defined on domain $\varOmega\subset \comp$. Let $\mul\in \bou(\bspace)$ be  a left-invertible multiplication operator  with a left-inverse $\lein\in \bou(\bspace)$, where $E$ is a Hilbert space. We assume that $ \mul$ satisfies the
following five properties:
\begin{itemize}
    \item $\bspace \hookrightarrow \hol(\varOmega, E)$ the inclusion map is 
continuous,
\item $\dim \nul(\lein)=n$,
\item every $f\in \nul(\lein)$ is analytic in an open neighborhood of $\bar\varOmega$,
\item for every $f\in \bspace$ and every $\lambda\in \varOmega$,
there exist $g\in \bspace$ and $h\in \nul(\lein)$
\begin{align*}
    f=(\mul-\lambda)g+h,
\end{align*}
\item $\sigma(\mul)=\bar{\varOmega}$.
\end{itemize}
This class includes the class described in Introduction, as well as certain adjoints of Cauchy dual operators for left-invertible operators (see the Shimorin analytic model \cite{S01}  and its extensions in \cite{ ja3}).

 We prove that the inverse $R_\lambda$ of $(I-\lambda\lein)$ has the following property: for every $f\in \mcal$ there exists $c_\lambda(f)\in \nul(\lein)$ such that
    \begin{align}\label{kluuu}
         (\mul-\lambda)Rf=\mul f-\lambda c_\lambda(f),
    \end{align}

 
Preliminary research show that this property is equivalent to the following one, expressed in terms of spectrum
\begin{align*}
    \lambda\in \sigma(\mul)\implies \frac{1}{\lambda}\notin \sigma(\lein).
\end{align*}

Observe that if we substitute $z=\lambda$ in \eqref{kluuu}, where $\lambda\in\varOmega$, $\lambda\neq0$ then we get $f(\lambda)=c_\lambda(f)(\lambda)$.
This makes it possible to define the following function
\begin{align}\label{przed}
\tilde{f} \colon D\ni\lambda \to c_\lambda(f)(\lambda)\in E,
\end{align}
where $D = \{ \lambda \in \comp \colon \frac{1}{\lambda} \notin \sigma(\lein|_\mcal)\}$. As before $\tilde{f}(\lambda)=f(\lambda)$, for $\lambda\in \varOmega$.
\begin{theorem}\label{iscd}
    Suppose that the sequence $\{\mul^n\lein^n\}$ converges to $0$ in the strong operator topology. Then $\lein$ is a Cowen-Douglas $\bscr_n(\varOmega^\prime)$, where $\varOmega^\prime=\{ \frac{1}{\lambda}\colon \lambda \notin \sigma(\mul)\}$.
\end{theorem}
\begin{theorem}\label{dopel}
Let $\mcal$ be an invariant subspace of $\lein$. Then the following conditions hold{\em :}
    \begin{align*}
      \sigmap&(\lein|_\mcal)\cap \{ \frac{1}{\lambda}\colon \lambda \notin \sigma(\mul)\} =\sigma_p(\lein|_\mcal)\cap \{ \frac{1}{\lambda}\colon \lambda \notin \sigma(\mul)\}\\&=\{\frac{1}{\lambda}\colon \text{there exist $h$ such that }(\mul-\lambda)^{-1}h\in \mcal \}
    \end{align*}
\end{theorem}

\begin{theorem}\label{brzeg}
Let $\bspace$  the following
 additional condition: 
If $f\in \bou(\bspace)$ is analytic in an open neighborhood of a point
 $\xi\in \partial \sigma(T)$, $\lambda_n\to \xi$ and
 \begin{equation}\label{additional_cond} 
     (\mul-\lambda)h_n=zf-\lambda_n c_{\lambda_n}(f)\text{ and }  (\mul-\xi)h=zf-\xi c_\xi(f)
     \end{equation}
     Then $h_n\to \infty$.
 Then 
 \begin{align*}
\sigma(\lein|_\mcal) \cap   \partial \sigma(\mul) =
     \partial \sigma(\mul) \setminus&
 \{1/\xi \in \sigma(\mul) \colon \text{ every } f \in \mcal\text{ extends }\\& \text{ to be analytic in a neighborhood of } \xi \}
 \end{align*}
\end{theorem}

Note that since $\lein$ is a (non-invertible) Cowen-Douglas operator the space  $\nul(\lein)$ is finite-dimensional. In the case where the operator $\lein$ is a backward shift the  space $\nul(\lein)$ is one-dimensional and consists of constant functions. However, in the case where $\lein$ is any Cowen-Douglas operator, the situation is more complicated. Some examples show that elements of $\nul(\lein)$ behave better than a regular function from space $\hh$,  that is, they expand beyond the $\varOmega$. For example, let $T$ be a left-invertible operator such that $T$ can be modelled as a multiplication operator on space of vector-valued holomorphic function on annulus. If $T^{\prime*}$ is a Cowen-Douglas operator and $h\in \nul(\lein)$. 
Note that, since $\sum_{n=0}^\infty (P_E {T^{\prime*n}} h) z^n = 0$, it follows from \cite{ja3} that the series for $h$ is as follows:
\begin{align*}
           U_h(z) =\sum_{n=1}^\infty
(P_ET^{n}h)\frac{1}{z^n}.
\end{align*}
Therefore this series is absolutely convergent outside some disc. 

\section{Proofs}

\begin{proof}[Proof of Theorem~\ref{dopel}]
Let $\lambda\in \comp\setminus\sigma(\mul)$, and 
take $h\in \nul(\lein)$. Then $(\mul-\lambda)(\mul-\lambda)^{-1}h=h$. Multiplying both sides by $\lein$, we obtain $(I-\lambda\lein)(\mul-\lambda)^{-1}h=0$. If $(\mul-\lambda)^{-1}h\in \mcal$  then it is an eigenvector of $\lein$ corresponding to the eigenvalue $\frac{1}{\lambda}$. Hence, $\frac{1}{\lambda}\in \sigma(\lein|_
\mcal)$.

Now, suppose $\frac{1}{\lambda}\in \sigmap(\lein|_\mcal)\cap \{ \frac{1}{\lambda}\colon \lambda \notin \sigma(\mul)\}$. Then, there exists a sequence  $\{f_n\}\subseteq \mcal$ such that $\|f_n\|=1$ and
\begin{align}\label{apwid}
    \lim_{n\to \infty} \|(I -\lambda\lein)f_n\|=0.
\end{align}

Note that $(\mul\lein-I)f_n\in \nul(\lein)$ and the sequence $\{(\mul\lein-I)f_n\}$ is bounded. Since $\nul(\lein)$ is finite-dimensional, by the Bolzano-Weierstrass theorem, we can,  without loss of generality assume that $\{(\mul\lein-I)f_n\}$ is convergent to some $h\in \nul(\lein)$. Using the identity
    \begin{align*}
       \mul(I -\lambda\lein)f_n=\lambda(I-\mul\lein)f_n+(\mul-\lambda)f_n.
    \end{align*}
    and taking the limit as $n\to \infty$, it follows from \eqref{apwid} that $\lim_{n\to\infty} \|(\mul-\lambda)f_n-h\|=0$. Since $\lambda\notin \sigma(\lein)$, we obtain $\frac{1}{\lambda}\in \sigma_{p}(\lein|_\mcal)$.
\end{proof}

\begin{lemma} Let $\mcal$ be an invariant subspace of $\lein$. Then the following condition holds{\em :}
\begin{itemize}
    \item[(i)]
    If $\frac{1}{\lambda}\notin\sigma(\lein|_\mcal)$ then the inverse of $(I-\lambda\lein)|_\mcal$ has the following property{\em :} for every $f\in \mcal$ there exist $c_\lambda(f)\in \nul(\lein)$ such that
    \begin{align}\label{klu}
         (\mul-\lambda)R_\lambda f=\mul f-\lambda c_\lambda(f),
    \end{align}
    \item[(ii)] If  $\frac{1}{\lambda}\notin\sigma(\lein)$ and $f=(\mul-\lambda)g+h$, where $h\in \nul(\lein)$,  then $g=(I-\lambda\lein)^{-1}\lein f$.
    \end{itemize}
\end{lemma}
\begin{proof}
(i)
Let $\lambda\in \mathbb{C}$ be such that $\frac{1}{\lambda}\notin \sigma(\lein|_\mcal)$. Then there exists $R_\lambda\in \boldsymbol{B}(\mcal)$ such that
\begin{align*}
    (I-\lambda\lein)|_\mcal R_\lambda f=f,\quad f\in \mcal.
\end{align*}
Multiplying both sides by $\mul$ and rearranging terms yields
\begin{align*}
    \mul f= (\mul-\lambda)R_\lambda f+\lambda(I-\mul\lein)R_\lambda f,\quad f\in \mcal.
\end{align*}
Note that $\lambda(I-\mul\lein)R_\lambda f\in \nul(\lein)$, since $\lein(I-\mul\lein)=\lein-\lein\mul\lein=0$.
Therefore, we conclude that for every for every $f\in \mcal$ there exist $c_\lambda(f)\in \nul(\lein)$ such that \eqref{klu} holds as claimed in part (i).

(ii) Suppose $f=(\mul-\lambda)g+h$, where $h\in \nul(\lein)$. Appying $\lein$ to both sides gives $\lein f=(I-\lambda\lein)g$. As $(I-\lambda\lein)$ is invertible by assumption, we can solve for $g$
to obtain $g=(I-\lambda\lein)^{-1}\lein f$.
\end{proof}

\begin{theorem}
    Let $\mul\in \bou(\bspace)$ be a left-invertible operator and $\lein\in \bou(\bspace)$ be its left-inverse. 
    Then the following conditions are equivalent:
    \begin{itemize}
        \item[(i)] 
        for every $f\in \bspace$ and every $\lambda\in \varOmega$, there exist $g\in \bspace$ and $h\in \nul(\lein)$ such that
\begin{align*}
    f=(\mul-\lambda)g+h,
\end{align*}
        \item[(ii)] $\lambda\in \varOmega\implies \frac{1}{\lambda}\notin \sigma(\lein)$.
    \end{itemize}
\end{theorem}
\begin{proof}
    (ii)$\Rightarrow$(i)
    Suppose that $\lambda \in \varOmega$ then $\frac{1}{\lambda} \notin \sigma(\lein)$, so the operator $I - \lambda \lein$ is invertible.
Define
\begin{align*}
     g:=(I-\lambda\lein)^{-1}\lein f\text{ and } h:=f-(\mul-\lambda)(I-\lambda\lein)^{-1}\lein f.
\end{align*}
   Since $\lein(\mul-\lambda)=(I-\lambda\lein)$ it implies that $h\in \nul(\lein)$.

    (i)$\Rightarrow$(ii) Fix $\lambda\in \varOmega$. Define operator $R_\lambda\in \bou(\bspace)$ by setting $R_\lambda f = g$, where $\mul f=(\mul-\lambda)g+h$ for some $h\in \nul(\lein)$. Therefore
    \begin{align*}
        \mul f = (\mul-\lambda)R_\lambda f +h.
    \end{align*}
 Applying  $\lein$ to both sides yields
 \begin{align*}
     f = (I-\lambda\lein)R_\lambda f.
 \end{align*}
 which shows that $R_\lambda$ is injective. To prove surjectivity, take any $g\in \bspace$ and define
 \begin{align*}
     h := \mul(I-\lambda\lein)g-(\mul-\lambda)g=-\lambda(\mul\lein-I)g.
 \end{align*}
 Then $h\in \nul(\lein)$. Putting $f := (I-\lambda\lein)g$, we get 
 \begin{align*}
     zf=(\mul-\lambda)g+h.
 \end{align*}
 Since $(I-\lambda\lein)R_\lambda=I$ and $R_\lambda$ is bijection it follows that $I-\lambda\lein$ is invertible. To prove that $R_\lambda$ is bounded, suppose that $\{f_n\}$ is a sequence in $\bspace$ such that $\lim_{n\to \infty}f_n=0$ and $\lim_{n\to \infty}R_\lambda f_n=g$ for some $g\in \bspace$. Then
 \begin{align*}
     0=\lim_{n\to \infty}f_n=\lim_{n\to \infty}(I-\lambda\lein)R_\lambda f_n=(I-\lambda\lein)g.
 \end{align*} Since $(I-\lambda\lein)$ is injective, we get that $g=0$. By the Closed Graph Theorem $R_\lambda$ is bounded.
\end{proof}

\begin{lemma}\label{nickry}
    Let $T\in \bou(\bspace)$ be a left-invertible but not invertible operator, and $L\in \bou(\bspace)$ be its left-inverse. 
    Then for every $\lambda \notin \sigma(T)$, we have $\frac{1}{\lambda} \in \sigma(L)$.
\end{lemma}
\begin{proof}
Suppose $\lambda \notin \sigma(T)$. Then $T - \lambda$ is invertible, so there exists $S \in \mathcal{B}(\mathcal{H})$ such that $(T-\lambda)S=I$ and $S(T-\lambda)=I$.  This implies that $(I-\lambda L)S=L(T-\lambda)S=L$.
Multiplying both sides on the left by $T$, we obtain
$(I-\lambda L)ST=I$. 
On the other hand, starting from $I = S(T - \lambda)$, we expand:
\begin{align*}
I &= S(T - \lambda TL + \lambda TL - \lambda) = ST(I - \lambda L) - \lambda(I - TL).
\end{align*}
Now, if $T$ were invertible, we would have $TL = I$, and the second term would vanish. However, since $T$ is not invertible, $TL \ne I$, so $I \ne ST(I - \lambda L)$. This combined with $(I-\lambda L)ST=I$ shows that $\frac{1}{\lambda}\in \sigma(L)$.
\end{proof}

\begin{lemma}\label{rpps}
Let 
\begin{align*}
    \mathcal{P}&=\bigvee\{\mul^kh\colon h\in \nul(\lein), k\in \zbb_+\}, \\\mathcal{S}&=\bigvee(\mul-\lambda)^{-1}h\colon \lambda\in \comp\setminus\sigma(\mul),  h\in \nul(\lein)\}.
\end{align*}
 Then $\mathcal{P}=\mathcal{S}$. Moreover, if the sequence $\{\mul^n\lein^n\}$ converges to $0$ in the strong operator topology, then $\bspace=\mathcal{P}=\mathcal{S}$.
\end{lemma}
\begin{proof}
 (i) To show $\mathcal{S}\subset\mathcal{P}$, fix $\lambda \in \comp\setminus\sigma(\mul)$ such that $|\lambda|>\|\mul\|$ and   $h\in \nul(\lein)$. Consider the polynomial 
    \begin{align*}
       p_n=-\sum_{k=0}^n \frac{1}{\lambda^{k+1}}\mul^kh.
    \end{align*}
We estimate the norm of the difference:
    \begin{align*}
        \|p_n-(\mul-\lambda)^{-1}h\|&=\|\frac{1}{\lambda^{n+2}}\mul^{n+1}(\mul-\lambda)^{-1}h\|
        \\&\leq \frac{1}{|\lambda|^{n+2}}\|\mul\|^{n+1}\|(\mul-\lambda)^{-1}h\|
    \end{align*}
    Since $|\lambda|>\|\mul\|$, we see that $\lim_{n\to \infty}p_n=(\mul-\lambda)^{-1}h$. Therefore, $\mathcal{S}\subset\mathcal{P}$.

 Now, we show $\mathcal{P}\subset\mathcal{S}$. We proceed by induction. First, let us note that the following limit holds:
\begin{align*}
  \lim_{\lambda\to \infty} (\mul-\lambda)^{-1}h = \lim_{\lambda \to \infty}(\frac{\mul}{\lambda}-1)^{-1}h= h,\quad h\in \nul(\lein).
\end{align*}
Therefore $\nul(\lein)\subset \mathcal{S}$.  Now suppose, that for $0<k<n$ for some $n \in \natu \cup \{0\}$, then for $\lambda \in \comp\setminus\sigma(\mul)$, we have
\begin{align*}
     \mul^{n+1}{(\mul-\lambda)^{-1}}h=\sum_{k=0}^n\lambda^{k}\mul^{n-k}h+\lambda^{n+1}{(\mul-\lambda)^{-1}}h,\quad h\in \nul(\lein), 
\end{align*}
 Furthermore,
\begin{align}\label{przeo}
    \| \lambda\mul^{n+1}{(\mul-\lambda)^{-1}}h-\mul^{n+1}h\|\leq  \| \mul^{n+2}{(\mul-\lambda)^{-1}}h\|
\end{align}
 But since $\lambda\to (\mul-\lambda)^{-1}h$ is an  analytic function, then
 the right-hand side of \eqref{przeo} goes to zero as $|\lambda|\to\infty$. Thus $\mul^{n+1}h\in \mathcal{S}$ for $h\in \nul(\lein)$ and so $\mathcal{P}\subset\mathcal{S}$.

(ii) Fix $f\in \bspace$. Observe that, since $\lein(I-\mul\lein)=0$, it follows that
\begin{align*}
    h_k:=(I-\mul\lein)\lein^kf\in \nul(\lein), \quad k\in \natu.
\end{align*} It now suffices to observe that we can write
    \begin{align*}
        f= \mul^n\lein^nf+\sum_{k=0}^{n-1}\mul^k(I-\mul\lein)\lein^kf=  \mul^n\lein^nf+\sum_{k=0}^{n-1}\mul^kh_k.
    \end{align*}
\end{proof}

\begin{proof}[Proof of Theorem~\ref{iscd}]
    It follows from Theorem~\ref{dopel} that $\varOmega^\prime\subset\sigma(\lein)$.

    Let $\frac{1}{\lambda}\in \varOmega$, so there exists $R_\lambda\in \bou(\bspace)$ such that $(\mul-\lambda)R_\lambda=I$. Multyplying both sides by $\lein$, we get $(I-\lambda\lein)R_\lambda=I$. Thus $\ran(\lein-\omega)=\bspace$.
    Combining Theorem~\ref{dopel} and Lemma~\ref{rpps} we see that $\bigvee_{\omega\in\varOmega}\ker(\lein-\omega)=\bspace$. It is a general fact that conditions (i) and (ii) imply that $\dim \ker (T-\omega)$ is constant (see \cite[p.~189]{CD78}).
\end{proof}
\begin{proof}[Proof of Theorem~\ref{brzeg}]

   Let $\frac{1}{\xi}\in \partial \sigma(\mul)$ such that $1/{\xi}\notin \sigma(\lein|_\mcal)$.
Then the following function is an analytic extension of $f$ in a neighborhood
of $\xi$
\begin{align*}
\tilde{f} \colon D\ni\lambda \to c_\lambda(f)(\lambda)\in E,
\end{align*}
where $D = \{ \lambda \in \comp \colon \frac{1}{\lambda} \notin \sigma(\lein|_\mcal)\}$. 

 Now suppose that for every  $f \in \mcal$  extends to be analytic in a neighborhood of  $\xi$. By the additional condition \eqref{additional_cond}, the function $g$ such that $(\mul-\xi)g=zf-\xi c_\xi(f)$
  is the norm limit of functions $g_n$ defined by 
  $(\mul-\xi)g_n=zf-\lambda_n c_{\lambda_n}(f)\in\mcal$
  for some $\lambda_n\in \varOmega$, thus $g\in \mcal$. We show that the inverse $R_\lambda$ of  $(I-\lambda\lein)|_\mcal$ is given by $R_\lambda f=g$, where $g$ is defined by the equation $ (\mul-\xi)g=zf-\xi c_\xi(f)$. We show that $R_\lambda$ is bounded.
   Suppose that $\{f_n\}$ is a sequence in $\mcal$ such that 
   $f_n \to 0$  and  $g_n\to h$,
   where $(\mul-\xi)g_n=zf_n-\xi c_\xi(f_n)$.
 Let $c\in\nul(\lein)$ be such that
 the sequence $\{c_\xi(f_n)\}$ converges to $c$.
 Thus
 \begin{align*}
  (\mul-\xi)h= -\xi c,
 \end{align*}
 which yields $0=(\mul-\xi)h+\xi c$. By uniqueness this 
 can only hold when $h = 0$.
 Thus by the closed graph theorem, $R_\lambda$ is a continuous operator on $\mcal$ and so $1/\xi\in \sigma(\lein|_\mcal)$.
\end{proof}
   \bibliographystyle{amsalpha}
   
   \end{document}